\documentclass[oneside]{amsart}
\usepackage[utf8]{inputenc}
\usepackage{amsmath,amstext,amsopn,amssymb, amsfonts, amsthm, mathtools}
\usepackage[dvipsnames]{xcolor}
\usepackage{geometry}
\usepackage{bbm}
\usepackage{pdfsync}
\usepackage{tikz}
\usepackage{enumitem}
\usepackage{subfig}
\usepackage{caption}
\usepackage{graphicx}
\usepackage{mathrsfs}
\usepackage[ruled,vlined,algosection]{algorithm2e}
\DeclareMathAlphabet{\mathpzc}{OT1}{pzc}{m}{it}

\usepackage[pdftex,bookmarks,colorlinks,breaklinks]{hyperref}  
\definecolor{dullmagenta}{rgb}{0.4,0,0.4}   
\definecolor{darkblue}{rgb}{0,0,0.4}
\definecolor{darkgreen}{rgb}{0,0.4,0}
\hypersetup{linkcolor=darkblue,citecolor=blue,filecolor=dullmagenta,urlcolor=darkblue} 

\newtheorem{TheoremLetter}{Theorem}
{}

\newtheorem*{definition*}{Definition}

\newtheorem{theorem}{Theorem}
\newtheorem*{theorem*}{Theorem}

\newtheorem*{conjecture*}{Conjecture}

\newtheorem*{question*}{Question}

\newtheorem*{lemma*}{Lemma}

\newtheorem*{corollary*}{Corollary}

\newtheorem*{remark*}{Remark}
\numberwithin{equation}{section}
\numberwithin{theorem}{section}

\makeatletter
\newcommand{\customlabel}[2]{%
   \protected@write \@auxout {}{\string \newlabel {#1}{{#2}{\thepage}{#2}{#1}{}} }%
   \hypertarget{#1}{#2}
}
\makeatother

\def\XXint#1#2#3{{\setbox0=\hbox{$#1{#2#3}{\int}$}
     \vcenter{\hbox{$#2#3$}}\kern-.5\wd0}}

\newcommand{\supp}{\operatorname{supp}}

\newcommand{\Enl}{\operatorname{Enl}}
\newcommand{\ind}{\mathbbm{1}}

\makeindex
\begin{document}

\author{Guillermo Rey}
\address{Universidad Autónoma de Madrid}
\email{guillermo.rey@uam.es}
\thanks{Research supported by grant PID2022-139521NA-I00, funded by MICIU/AEI/10.13039/501100011033}

\title{Merryfield's inequality for multiparameter martingales}
\maketitle
\begin{abstract}
  We extend an inequality of Merryfield, valid in the continuous setting, to discrete multiparameter martingales.
  As a consequence, we obtain the $L^p$ comparison of the maximal function with the square function:
  \begin{align*}
    E[(Sf)^p] \lesssim E[(f^*)^p]
  \end{align*}
  for regular multiparameter filtrations and $0 < p < \infty$.
\end{abstract}

\section{Introduction}

The Burkholder-Davis-Gundy inequality establishes the equivalence of two ways of defining the martingale Hardy space: through the maximal function and through the square function.
See for example \cite{MR400380}.
Specifically, for every martingale $f$, one has
\begin{align} \label{BDG}
  \|Sf\|_{L^p} \sim \|f^*\|_{L^p}
\end{align}
for all $p \geq 1$, where $f^*$ is Doob's maximal function and $Sf$ is the martingale square function.
In this article, we will work exclusively in the discrete setting, in which case these operators take the form
\begin{align*}
  f^* = \sup_{m \geq 0} |f_m|, \quad  Sf = \bigg( \sum_{m = 0}^\infty (\Delta f_m)^2 \bigg)^{\frac{1}{2}}, \quad\text{and}\quad \Delta f_m = \begin{cases}
    f_0 &\text{if } m = 0, \\
    f_m - f_{m-1} &\text{if } m \geq 1.
  \end{cases}
\end{align*}

When one works in the \emph{biparameter} setting, where the filtration and martingales are doubly indexed, the equivalence is more complicated.
In this situation, Doob's maximal function is defined similarly:
\begin{align*}
  f^* = \sup_{m,n \geq 0} |f_{m,n}|.
\end{align*}
The biparameter square function
\begin{align*}
  Sf = \bigg( \sum_{m,n \geq 0} (\Delta f_{m,n})^2 \bigg)^{\frac{1}{2}}
\end{align*}
involves the biparameter difference sequence $\Delta f_{m,n}$. When interpreted as an operator acting on sequences, $\Delta$ is the composition $\Delta_1 \circ \Delta_2 = \Delta_2 \circ \Delta_1$ of the two directional difference operators:
\begin{align*}
  \Delta_1 f_{m,n} = \begin{cases}
    f_{0, n} &\text{if } m = 0, \\
    f_{m,n} - f_{m-1, n} &\text{if } m \geq 1,
  \end{cases}
  \quad\text{and}\quad
  \Delta_2 f_{m,n} = \begin{cases}
    f_{m, 0} &\text{if } n = 0, \\
    f_{m,n} - f_{m, n-1} &\text{if } n \geq 1.
  \end{cases}
\end{align*}
In particular, when $m$ and $n$ are greater than $1$, we have $\Delta f_{m,n} = f_{m,n} - f_{m-1,n} - f_{m,n-1} + f_{m,n}$.

We are interested in the multiparameter version of the Burkholder-Davis-Gundy inequality.
The operators involved in the multiparamter case can be defined in a way similar to how the two-parameter operators were defined above,
and this will be done in full generality in Section \ref{section.multiparameter}.
To keep the notation light, in the introduction we will focus on the two-parameter case.

When $p > 1$, one can use vector-valued versions of the one-parameter Burkholder-Davis-Gundy inequalities to obtain the multiparameter ones, see \cite{MR2500519}.
However, when $p=1$ these arguments cannot be easily applied, and much less is known without additional hypotheses.

The $E[f^*] \lesssim E[Sf]$ side of the inequality can be shown for \emph{regular} filtrations satisfying the Cairoli-Walsh $(F_4)$ condition.
The argument, due to C. Fefferman, can be found in \cite{MR539351} for dyadic martingales,
and \cite{MR630308} for martingales with respect to a general regular filtration.
We say that a two-parameter filtration $\mathcal{F}_{m,n}$ is \emph{regular} if there exists a finite constant $R$ such that
\begin{align*}
  a_{m+1,n} \leq R a_{m,n} \quad\text{and}\quad a_{m,n+1} \leq R a_{m,n}
\end{align*}
for every positive biparameter martingale $a$, and every $m,n \geq 0$. For the dyadic filtration, this property is sometimes called the \emph{doubling} condition.

The \eqref{intro:f4} condition, introduced by Cairoli and Walsh in \cite{MR420845}, guarantees that $\mathcal{F}_{m+1,n}$ and $\mathcal{F}_{m,n+1}$ are conditionally independent given $\mathcal{F}_{m,n}$.
Equivalently:
\begin{align} \label{intro:f4} \tag{$F_4$}
  \mathbb{E}[ \mathbb{E}[f \,|\, \mathcal{F}_{m,n}] \,|\, \mathcal{F}_{m',n'}] = \mathbb{E}[f \,|\, \mathcal{F}_{\min(m,m'), \min(n,n')}] \quad \text{for all $m,n,m',n' \geq 0$},
\end{align}
where $\mathbb{E}[f \,|\, \mathcal{A}]$ denotes the conditional expectation of $f$ with respect to a sub $\sigma$-algebra $\mathcal{A}$.

Proving the reverse inequality $E[Sf] \lesssim E[f^*]$ for biparameter martingales turned out to be considerably more difficult. This was shown by Brossard in \cite{MR602392}
(see also \cite{MR1320508} for a more detailed explanation).
In particular, Brossard proved that
\begin{align} \label{intro.brossards_ineq}
  P(Sf > \lambda) \lesssim P(f^* \lambda) + \lambda^{-2}\mathbb{E}\Big[ f^*;\, (f^*)^2 \leq \lambda \Big]
\end{align}
holds for two-parameter martingales $f$ with respect to a regular filtration. Applying the layer-cake formula yields $E[Sf] \lesssim E[f^*]$\footnote{It is currently unknown whether \eqref{BDG} holds in the biparameter case when the filtration is not regular.}.
We note that under these hypotheses one has $E[(f^*)^p] \sim E[(Sf)^p]$ also for $0 < p \leq 1$.

A similar question can be asked in the continuous setting, where the maximal function $f^*$ is replaced by the non-tangential maximal function
\begin{align*}
  Nu(x) = \sup_{(y,t) \in \Gamma(x)} |u(y, t)|,
\end{align*}
defined for functions $u$ on the $n$-parameter upper half space $\mathbb{R}^n \times \mathbb{R}_{+}^n$, and $u$ is $n$-harmonic: a harmonic function of $(x_i, t_i)$ for each $1 \leq i \leq n$.
Here $\Gamma(x)$ denotes the cone centered at $x$: $\Gamma(x) = \{(y,t):\, |y_i - x_i| \leq t_i \text{ for all $1 \leq i \leq n$}\}$.

In the contiunous setting, the square function is replaced by the \emph{area integral}:
\begin{align*}
  Au(x) = \Bigg( \int_{\Gamma(x)} |\nabla_1 \dots \nabla_n u (y,t)|^2 \, dt\,dy \Bigg)^{\frac{1}{2}}.
\end{align*}
Here $\nabla_i$ is the vector-valued operator $(\partial_{x_i}, \partial_{t_i})$ and $\nabla_1 \dots \nabla_n$ is the tensor product of all $n$ operators.

With this defintion, one has
\begin{align} \label{gundy_stein}
  \|Au\|_{L^p} \lesssim \|Nu\|_{L^p}
\end{align}
for $0 < p < 2$. This was shown by Gundy and Stein in \cite{MR524328}. Like in the martingale setting, the reverse inequality is easier.

In \cite{MR794581}, Merryfield gave a real-variable proof of \eqref{gundy_stein}.
In particular, Merryfield deduces \eqref{gundy_stein} as a consequence of the following ``one-parameter'' result:
\begin{theorem}[Lemma 3.1 in \cite{MR794581}] \label{merryfield_thm}
  Suppose $u$ is harmonic on $\mathbb{R}^2_+$ and continuous on its closure. Suppose that $f$ is a function such that $f-k_f \in L^2(\mathbb{R})$ for some constant $k_f$,
  $Nu$ is bounded on $\supp f$, and either $Nu \in L^2$ or $f \in L^2$. Then
  \begin{align*}
    \int_{\mathbb{R}^2_+} |\nabla u(x,t)|^2 |\Phi_t \ast f(x)|^2 \, t \, dt \, dx \leq
    \int_{\mathbb{R}} |u(x,0)|^2 |f(x)|^2 \, dx + \int_{\mathbb{R}^2_+} |u(x,t)|^2 |\Psi_t \ast f(x)|^2 t^{-1} \, dt\, dx,
  \end{align*}
  where $\Psi$ is a vector-valued function in $C_c(\mathbb{R})$ with the same support as $\Phi$ and mean value $0$.
\end{theorem}
We note that Merryfield's lemma has been extended to the setting
of product Hardy spaces on stratified groups in \cite{Cowling2025}.
However, it seems non-trivial to extend their proof to the martingale setting, or to more than two parameters, see the comment at the end their ``Part 1 of the proof of Proposition 5.2'' in \cite{Cowling2025}.

The purpose of this article is to obtain a result similar to Merryfield's in the setting of multiparameter martingales.
As immediate corollary, we can extend Brossard's result $E[(Sf)^p] \lesssim E[(f^*)^p]$ to arbitrarily many parameters.

As in Merryfield's article, the multiparameter result will follow from iteration of the single-parameter case.
The single-parameter case does not require extra hypotheses on the filtration, so we state it separately:

\begin{TheoremLetter} \label{TheoremA}
  Suppose that $(\mathcal{F}_m)_{m \geq 0}$ is a filtration in a probability space.
  Let $f$ and $a$ be martingales with respect to this filtration, and assume furthemore that $f_m a_n \in L^2$ for all $m,n \geq 0$.
  Then for every $M \geq 1$ we have
  \begin{align*}
    \mathbb{E}\Bigg( \sum_{m=1}^M (\Delta f_m)^2 a_{m-1}^2 \,\Big|\, \mathcal{F}_0 \Bigg) \leq
      20\mathbb{E}\Bigg( f_M^2 a_M^2 + \sum_{m=1}^M (f_m^2 + f_{m-1}^2)(\Delta a_m)^2 \,\Big|\, \mathcal{F}_0 \Bigg).
  \end{align*}
\end{TheoremLetter}

The proof of Theorem \ref{TheoremA} is based on Merryfield's proof
in \cite{MR794581}, but some non-trivial modifications are needed.
We explain in detail what needs to be changed in Section \ref{section.one_parameter}. In short: in the continuous setting
one integrates by parts an expression of the form
\begin{align*}
\int_{\mathbb{R}^2_+} (\nabla u \cdot \nabla u) (\text{other terms}).
\end{align*}
After applying the product rule to $\nabla u (\text{other terms})$, the
part where the derivatives hit $\nabla u$ becomes the Laplacian, which disappears since $u$ is harmonic. In the martingale setting this is no longer the case, and we have to proceed differently.

Iterating Theorem \ref{TheoremA} we can obtain a multiparameter version.
In this case, we require the filtration to satisfy the \eqref{intro:f4} condition.
\begin{TheoremLetter} \label{TheoremB}
  Let $\mathcal{F}$ be a $k$-parameter filtration satisfying the \eqref{f4} condition.
  Let $f$ and $a$ be $k$-parameter martingales.
  Suppose furthermore that $f_{m} a_{n} \in L^2$ for every $m,n \in \mathbb{N}^k$.
  Then, for every $M \in \mathbb{N}^k$ we have
  \begin{align*}
    \mathbb{E} \Bigg( \sum_{1 \leq m \leq M} (\Delta f_{m})^2 a^2_{m - 1} \Bigg) \lesssim \mathbb{E} \Bigg( \sum_{\mathcal{I} \subseteq [k]} \sum_{m \in \partial_{\mathcal{I}}(M)} (f^*_{m})^2 (\Delta_{\mathcal{I}} a_{m})^2 \Bigg).
  \end{align*}
  We will give precise definitions in Section \ref{section.multiparameter}, but briefly: the sum on the right hand side is taken over all subsets $\mathcal{I}$ of $\{1,2,\dots, k\}$
  and over the integer-lattice box $\partial_{\mathcal{I}}(M) \subseteq \{1 \leq m \leq M\}$ consisting on keeping all coordinates but those in $\mathcal{I}$ fixed to the upper corner $M$.
\end{TheoremLetter}

If, furthremore, we assume that the filtration is regular, then we can obtain Brossard's theorem in arbitrarily many parameters:
\begin{TheoremLetter} \label{TheoremC}
  For any $k$-parameter martingale $f$ with respect to a regular filtration satisfying the \eqref{f4} condition, we have
  \begin{align} \label{intro.comp}
    \mathbb{E}[(Sf)^p] \lesssim \mathbb{E}[(f^*)^p]
  \end{align}
  for all $0 < p < \infty$.
\end{TheoremLetter}

When $p > 1$ this is well known. We will show that, when $p < 2$, Theorem \ref{TheoremC} can be deduced from Theorem \ref{TheoremB}.
We will follow the ideas of the analogous proofs in \cite{MR602392} and \cite{MR794581} and which are based on
proving \eqref{intro.brossards_ineq}, from which \eqref{intro.comp} follows after applying the layer cake formula.

One can use this inequality to prove a global endpoint estimate of the form
\begin{align*}
P(Sf > \lambda) \lesssim \mathbb{E}\Bigg[ \frac{|f|}{\lambda} \log\bigg( e + \frac{|f|}{\lambda}\bigg)^{k-1} \Bigg].
\end{align*}
This also follows from the layer cake formula, as shown in \cite{cmylp}. We skip these details, but refer the reader to the proof of Theorem 4.2 in their article.


\section{One-parameter martingales}
\label{section.one_parameter}
Let $(\Omega, \Sigma, P)$ be a probability space. In this section, we will assume that we have a one-parameter discrete filtration $(\mathcal{F}_{m})_{m \geq 0}$,
i.e.: a family of sub $\sigma$-algebras $\mathcal{F}_m \subseteq \Sigma$ indexed by $m \in \mathbb{Z}_{\geq 0}$ and satisfying
\begin{align*}
  \mathcal{F}_{m+1} \subseteq \mathcal{F}_m
\end{align*}
for all $m \geq 0$. We will denote by $\mathcal{F}_\infty$ the union
\begin{align*}
  \mathcal{F}_{\infty} = \bigcup_{m \geq 0} \mathcal{F}_m.
\end{align*}
We will \emph{not} assume that $\mathcal{F}_\infty$ necessarily coincides with $\Sigma$, or that $\mathcal{F}_0$ is trivial.

A martingale is a sequence $(f_m)_{m \geq 0}$ of integrable functions such that $f_m$ is $\mathcal{F}_m$-measurable and
\begin{align*}
  \mathbb{E}[f_{m+1} \,|\, \mathcal{F}_m] = f_m
\end{align*}
for all $m \geq 0$.

We now define the difference operator $\Delta$, which which we understand as acting on sequences. Specifically, if $f = (f_m)_{m \geq 0}$ then $\Delta f$ is another sequence given by
\begin{align*}
  (\Delta f)_m = \begin{cases}
    f_0 &\text{if } m = 0 \\
    f_m - f_{m-1} &\text{if } m \geq 1.
  \end{cases}
\end{align*}
Note that, for this definition, we do not need $f$ to be a sequence of \emph{functions}, $\Delta$ acts on sequences taking values in any vector space. This wider applicability will be useful later when we use the summation by parts formula.

We can now state the main theorem of this section.
\begin{theorem} \label{Theorem1}
  Suppose that $(\mathcal{F}_m)_{m \geq 0}$ is a filtration with respect to $(\Omega, P)$.
  Let $f$ and $a$ be martingales with respect to this filtration, and assume furthemore that $f_m a_n \in L^2$ for all $m,n \geq 0$.
  Then for every $M \geq 1$ we have
  \begin{align}
    \mathbb{E}\Bigg( \sum_{m=1}^M (\Delta f_m)^2 a_{m-1}^2 \,\Big|\, \mathcal{F}_0 \Bigg) \leq
      20\mathbb{E}\Bigg( f_M^2 a_M^2 + \sum_{m=1}^M (f_m^2 + f_{m-1}^2)(\Delta a_m)^2 \,\Big|\, \mathcal{F}_0 \Bigg).
  \end{align}
\end{theorem}
\begin{proof}
  We would like to apply summation by parts to the sum in the left-hand-side. However, the square is preventing us from doing so.
  One way to get around this obstacle is by writing
  \begin{align*}
    \sum_{m=1}^M (\Delta f_m)^2 a_{m-1}^2 &= \sum_{m=1}^M (\Delta f_m) (\Delta f_m) a_{m-1}^2 \\
    &= \text{(boundary terms)} - \sum_{m=0}^M f_m \Delta( (\Delta f)a_{\cdot -1}^2 )_m.
  \end{align*}
  This is, essentially, the approach that Merryfield takes in \cite{MR794581}. The issue with this idea is that, in our setting, we pick up a double difference factor $\Delta^2 f$.
  In Merryfield's article, these double differences correspond to the Laplacian of $f$, which disappears since the functions there are harmonic.
  We must take a different route since our setting does not enjoy that simplification.

  Recall, from elementary probability, the variance identity $\mathbb{E}[(f - \mathbb{E}f)^2] = \mathbb{E}[f^2] - (\mathbb{E}f)^2$.
  One can easily generalize this formula to conditional expectations:
  \begin{align} \label{variance_identity}
    \mathbb{E}[(\Delta f_m)^2 \,|\, \mathcal{F}_{m-1}] = \mathbb{E}[f_m^2 - f_{m-1}^2 \,|\, \mathcal{F}_{m-1}] = \mathbb{E}[ \Delta (f^2)_m \,|\, \mathcal{F}_{m-1}]
  \end{align}
  for $m \geq 1$.

  Here, $\Delta(f^2)$ is the sequence given by the difference operator applied to $f^2$, which is the sequence defined pointwise by $(f^2)_m = f_m^2$.
  In particular, note that $f^2$ is \emph{not} a martingale, and thus $\Delta(f^2)$ is not a martingale difference sequence.
  This, however, will not be a problem for us.

  Let us recall the summation-by-parts framework in our setting. Define the shift operator $B$, acting on sequences, by
  \begin{align*}
    (Bf)_m = \begin{cases}
      0 &\text{if } m = 0 \\
      f_{m-1} &\text{if } m \geq 1.
    \end{cases}
  \end{align*}
  We then have $\Delta f = f - Bf$. The product rule takes the following form
  \begin{align*}
    \Delta(fg) = f\Delta g + (\Delta f)Bg.
  \end{align*}
  We can use this formula, together with the identity $\sum_{m=0}^M (\Delta f)_m = f_M$ to arrive at the summation by parts formula that we will use:
  \begin{align*}
    \sum_{m=0}^M (\Delta f)_m (Bg)_m = f_M g_M - \sum_{m=0}^M f_m (\Delta g)_m.
  \end{align*}
  We will omit the dependence on $m$ when it does not introduce confusion, and denote summation over $\{a, \dots, b\}$ by $\sum_{[a, b]}$.
  Then the above formula takes the form
  \begin{align} \label{sbp}
    \sum_{[0,M]} (\Delta f)Bg = f_M g_M - \sum_{[0,M]} f \Delta g.
  \end{align}

  Proceeding with the proof, we use the conditional variance identity \eqref{variance_identity} together with the fact that $a_{m-1}^2$ is $\mathcal{F}_{m-1}$-measurable.
  \begin{align*}
    \mathbb{E} \Bigg( \sum_{m=1}^M (\Delta f_m)^2 a_{m-1}^2 \,\Big|\, \mathcal{F}_{0} \bigg) &= \mathbb{E}\Bigg( \sum_{m=1}^M \mathbb{E}\Big( (\Delta f_m)^2 a_{m-1}^2 \,|\, \mathcal{F}_{m-1} \Big) \,\Big|\, \mathcal{F}_0 \Bigg) \\
    &= \mathbb{E}\Bigg( \sum_{m=1}^M \mathbb{E}\Big( \Delta (f^2)_m a_{m-1}^2 \,|\, \mathcal{F}_{m-1} \Big) \,\Big|\, \mathcal{F}_0 \Bigg) \\
    &= \mathbb{E}\Bigg( \sum_{[1,M]} \Delta(f^2) (Ba)^2 \,\Big|\, \mathcal{F}_0 \Bigg).
  \end{align*}
  Here we have used that $B$ is multiplicative, so in particular $B(f^2) = (Bf)^2$. We will use this in what follows without further mention.

  Now, we use the summation by parts formula \eqref{sbp}:
  \begin{align*}
    \sum_{[1,M]} \Delta(f^2) (Ba)^2 = f_M^2 a_M^2 - \sum_{[0,M]} f^2 \Delta(a^2).
  \end{align*}
  It would be convenient to use the same technique now to pull the square out of $\Delta(a^2)$,
  but the issue is that $\Delta(a^2)$ is not being multiplied by a \emph{predictable} sequence like in the previous situation.
  Instead, we can write
  \begin{align*}
    f_M^2 a_M^2 - \sum_{[0,M]} f^2 \Delta(a^2) &= f_M^2 a_M^2 - f_0^2 a_0^2  - \sum_{[1,M]} f^2 \Delta(a^2) \\
    &\leq f_M^2 a_M^2 - \sum_{[1,M]}(f^2 - Bf^2)\Delta(a^2) - \sum_{[1,M]}(Bf^2)\Delta(a^2) \\
    &= f_M^2 a_M^2 - \sum_{[1,M]}\Delta(f^2)\Delta(a^2) - \sum_{[1,M]}(Bf^2)\Delta(a^2).
  \end{align*}
  Now $Bf^2$ \emph{is} predictable, so by \eqref{variance_identity} and using that $a$ is also a martingale we obtain
  \begin{align*}
    \mathbb{E}\Bigg( - \sum_{[1,M]}(Bf^2)\Delta(a^2) \,\Big|\, \mathcal{F}_0 \Bigg) &= \mathbb{E}\Bigg( - \sum_{[1,M]}(Bf)^2(\Delta a)^2 \,\Big|\, \mathcal{F}_0 \Bigg) \leq 0.
  \end{align*}

  We have thus arrived at the inequality
  \begin{align} \label{1p::intermediate}
    \mathbb{E}\Bigg( \sum_{[1,M]} (\Delta f)^2 Ba^2 \,\Big|\, \mathcal{F}_0 \Bigg) \leq \mathbb{E}\Bigg( f_M^2 a_M^2 - \sum_{[1,M]} \Delta (f^2) \Delta(a^2) \,\Big|\, \mathcal{F}_0 \Bigg).
  \end{align}

  The difference operator applied to the square of a sequence takes the particularly simple form $\Delta(f^2) = (\Delta f)(f + Bf)$.
  We can use this identity to bound each term in the sum:
  \begin{align*}
    |\Delta (f^2) \Delta(a^2)| &= |\Delta(f^2)(\Delta a)(a + Ba)| \\
    &= |\Delta(f^2)(\Delta a)(\Delta a + 2Ba)| \\
    &\leq |\Delta (f^2)(\Delta a)^2| + 2|\Delta(f^2)(\Delta a)Ba| \\
    &\leq (f^2 + Bf^2)(\Delta a)^2 + 2|(\Delta f)(f + Bf)(\Delta a)Ba|.
  \end{align*}
  The first term is already what appears on the right hand side of the inequality that we want to prove.
  For the second term we use Young's inequality in the form $2|xy| \leq x^2 + y^2$:
  \begin{align*}
    2|(\Delta f)(f + Bf)(\Delta a)Ba| &\leq \frac{1}{2}(\Delta f)^2 Ba^2 + 2 (f+Bf)^2 (\Delta a)^2 \\
    &\leq \frac{1}{2}(\Delta f)^2 Ba^2 + 4 (f^2+Bf^2) (\Delta a)^2.
  \end{align*}
  Putting it all together, we have
  \begin{align*}
    |\Delta (f^2) \Delta(a^2)| &\leq |(f^2 + Bf^2)(\Delta a)^2| + \frac{1}{2}(\Delta f)^2 Ba^2 + 4 (f^2+Bf^2) (\Delta a)^2 \\
    &\leq 5(f^2+Bf^2) (\Delta a)^2 + \frac{1}{2}(\Delta f)^2 Ba^2.
  \end{align*}

  Plugging this into \eqref{1p::intermediate} and subtracting $\frac{1}{2}(\Delta f)^2 Ba^2$ on both sides, we arrive at
  \begin{align*}
    \mathbb{E}\Bigg( \sum_{[1,M]} (\Delta f)^2 Ba^2 \,\Big|\, \mathcal{F}_0 \Bigg) \leq \mathbb{E}\Bigg( 2f_M^2 a_M^2 + 10\sum_{[1,M]} (\Delta a)^2 (f^2+Bf^2) \,\Big|\, \mathcal{F}_0 \Bigg).
  \end{align*}

\end{proof}

\section{Multiparameter martingales}
\label{section.multiparameter}

In this section, we prove a multiparameter version of Theorem \ref{Theorem1}.
The proof requires a lighter notation in the two-parameter setting, so we will prove that case first and then, in the next subsection, the case of an arbitrary number of parameters.
Both proofs are essentially the same, so the arguments will be repeated. We hope this is a small price to pay in the benefit of exposition.

\subsection{Two parameters}
As in the previous section, let $(\Omega, \Sigma, P)$ be a probability space.
A two-parameter filtration is a doubly-indexed sequence $(\mathcal{F}_{m,n})_{m,n \geq 0}$ of sub $\sigma$-algebras of $\Sigma$ which is increasing in each parameter:
\begin{align*}
  \mathcal{F}_{m,n} \subseteq \mathcal{F}_{m+1,n} \quad \text{and} \quad \mathcal{F}_{m,n} \subseteq \mathcal{F}_{m,n+1}.
\end{align*}

We will assume that our filtration satisfies the \eqref{f4_2p} condition which, we recall, can be stated as the identity
\begin{align} \label{f4_2p} \tag{$F_4$}
  \mathbb{E}[ \mathbb{E}[f \,|\, \mathcal{F}_{m,n}] \,|\, \mathcal{F}_{m',n'}] = \mathbb{E}[f \,|\, \mathcal{F}_{\min(m,m'), \min(n,n')}] \quad \text{for all $m,n,m',n' \geq 0$}.
\end{align}

A two-parameter martingale (with respect to the $\mathcal{F}_{m,n}$ filtration) is a sequence $(f_{m,n})_{m,n \geq 0}$ of integrable functions such that each $f_{m,n}$ is $\mathcal{F}_{m,n}$-measurable,
and such that
\begin{align*}
  \mathbb{E}[f_{m+1,n} \,|\, \mathcal{F}_{m,n}] = \mathbb{E}[f_{m,n+1} \,|\, \mathcal{F}_{m,n}] = f_{m,n}.
\end{align*}

The two-paremeter difference operator is the composition of the directional difference operator in each direction.
We introduce these first:
\begin{align*}
  \Delta_1 f_{m,n} = \begin{cases}
    f_{0,n} &\text{if } m = 0, \\
    f_{m,n} - f_{m,n} &\text{if } m \geq 1.
  \end{cases} \quad\quad
  \Delta_2 f_{m,n} = \begin{cases}
    f_{m,0} &\text{if } n = 0, \\
    f_{m,n-1} - f_{m,n} &\text{if } n \geq 1.
  \end{cases}
\end{align*}
It is easy to see that these directional-difference operators commute with each other, and we define the ``full'' difference operator as their composition:
\begin{align*}
  \Delta = \Delta_1 \circ \Delta_2 = \Delta_2 \circ \Delta_1.
\end{align*}
In particular, when $m$ and $n$ are both greater than $1$, we have
\begin{align*}
  \Delta f_{m,n} = f_{m,n} - f_{m-1,n} - f_{m,n-1} + f_{m-1,n-1}.
\end{align*}

We can now begin the generalization of Theorem \ref{Theorem1} to two parameters. Assume $f$ and $a$ are two-parameter martingales, we want to establish an upper bound for
the expected value of
\begin{align*}
  \sum_{\substack{1 \leq m \leq M \\ 1 \leq n \leq N}} (\Delta f_{m,n})^2 a_{m-1,n-1}^2.
\end{align*}
We can write this experession as
\begin{align*}
  \sum_{m=1}^M \sum_{n=1}^N (\Delta_2(\Delta_1 f)_{m,n})^2 a_{m-1,n-1}^2.
\end{align*}
Now, define $\widetilde{f}_{m,n} = \Delta_1 f_{m,n}$ and $\widetilde{a}_{m,n} = a_{m-1,n}$. We can re-write the above expression as
\begin{align*}
  \sum_{m=1}^M \sum_{n=1}^N (\Delta_2 \widetilde{f}_{m,n})^2 (\widetilde{a}_{m,n-1})^2.
\end{align*}

Now, we want to interpret this a one-parameter martingale and apply Theorem \ref{Theorem1}. To this end, define the sequence of $\sigma$-alegbras
\begin{align*}
  \mathcal{G}_n^1 = \sigma\Bigg( \bigcup_{m=0}^\infty \mathcal{F}_{m,n}\Bigg).
\end{align*}
This forms a one-parameter filtration. Also, one can see that for every fixed $m$, the sequences $(\widetilde{f}_{m,n})_{n \geq 0}$ and $(\widetilde{a}_{m,n})_{n \geq 0}$ are martingales with respect to $(\mathcal{G}_n^1)_{n \geq 0}$.
Indeed, for all $j \geq 0$, we can use the \eqref{f4_2p} condition twice to obtain:
\begin{align*}
  \mathbb{E}[f_{m,n+1} \,|\, \mathcal{F}_{j,n}] &= \mathbb{E}\Big( \mathbb{E}[f_{m,n+1} \,|\, \mathcal{F}_{m,n+1}] \,|\, \mathcal{F}_{j,n} \Big) \\
  &= \mathbb{E}\Big( f_{m,n+1} \,|\, \mathcal{F}_{\min(m,j), n} \Big) \\
  &= \mathbb{E}\Big( \mathbb{E}(f_{m,n+1} \,|\, \mathcal{F}_{m,n}) \,|\, \mathcal{F}_{\min(m,j), n} \Big) \\
  &= \mathbb{E}\Big( f_{m,n} \,|\, \mathcal{F}_{\min(m,j), n} \Big).
\end{align*}
In particular,
\begin{align*}
  \mathbb{E}[f_{m,n+1} \,|\, \mathcal{F}_{j,n}] = f_{m,n} \quad \text{for all $j \geq m$} \implies \mathbb{E}[f_{m,n+1} \,|\, \mathcal{G}_n^1] = f_{m,n}.
\end{align*}
The same can be done with $f_{m-1,n}$, and thus $\widetilde{f}$ is a martingale with respect to $\mathcal{G}^1$. A similar argument works for $\widetilde{a}$.

Now, we can write
\begin{align*}
  \mathbb{E}\Bigg( \sum_{m=1}^M \sum_{n=1}^N (\Delta_2 \widetilde{f}_{m,n})^2 (\widetilde{a}_{m,n-1})^2 \Bigg) &= \mathbb{E}\Bigg( \sum_{m=1}^M \mathbb{E} \Bigg[ \sum_{n=1}^N  (\Delta_2 \widetilde{f}_{m,n})^2 (\widetilde{a}_{m,n-1})^2 \,\Big|\, \mathcal{G}_0^1 \Bigg] \Bigg).
\end{align*}
If we now apply Theorem \ref{Theorem1} to the inner conditional expectation we obtain
\begin{align*}
  \mathbb{E}\Bigg( \sum_{m=1}^M \sum_{n=1}^N (\Delta_2 \widetilde{f}_{m,n})^2 (\widetilde{a}_{m,n-1})^2 \Bigg) &\lesssim
    \mathbb{E}\Bigg(\sum_{m=1}^M \mathbb{E}\Bigg[ \widetilde{f}_{m,N}^2 \widetilde{a}_{m,N}^2 + \sum_{n=1}^N (\widetilde{f}_{m,n}^2 + \widetilde{f}_{m,n-1}^2) (\Delta_2 \widetilde{a}_{m,n})^2 \,\Big|\, \mathcal{G}_0^1 \Bigg]\Bigg).
\end{align*}

Let us analyze each term separately. The first term can be re-written in terms of the original $f$ and $a$ as follows:
\begin{align*}
  \mathbb{E} \Bigg( \sum_{m=1}^M \mathbb{E}\Big[ (\Delta_1 f_{m,N})^2 a_{m-1,N}^2 \,|\, \mathcal{G}_0^1 \Big] \Bigg) = 
    \mathbb{E} \Bigg( \sum_{m=1}^M (\Delta_1 f_{m,N})^2 a_{m-1,N}^2 \Bigg).
\end{align*}
If we fix $N$, $f_{m,N}$ and $a_{m,N}$ are martingales with respect to the filtration $(\mathcal{F}_{m,N})_{m \geq 0}$.
Thus, we can apply Theorem \ref{Theorem1} to this expectation and obtain
\begin{align*}
  \mathbb{E} \Bigg( \sum_{m=1}^M (\Delta_1 f_{m,N})^2 a_{m-1,N}^2 \Bigg) &\lesssim \mathbb{E} \Bigg( f_{M,N}^2 a_{M,N}^2 + \sum_{m=1}^M (f_{m,N}^2 + f_{m-1,N}^2) (\Delta_1 a_{m,N})^2 \Bigg) \\
  &\lesssim \mathbb{E} \Bigg( f_{M,N}^2 a_{M,N}^2 + \sum_{m=1}^M (f_{m,N}^*)^2 (\Delta_1 a_{m,N})^2 \Bigg).
\end{align*}
Here we have used the ``stopped'' version of Doob's maximal function:
\begin{align*}
f_{m,n}^* = \sup_{\substack{i \leq m \\ j \leq n}} |f_{i,j}|,
\end{align*}
though it is clear that we could have used the smaller
\begin{align*}
f_{m,n}^{b} = \sup_{\substack{m-1 \leq i \leq m \\ n-1 \leq j \leq n}} |f_{i,j}|.
\end{align*}

Next, we analyze the term
\begin{align*}
  \mathbb{E} \Bigg( \sum_{m=1}^M \mathbb{E} \Bigg( \sum_{n=1}^N \widetilde{f}_{m,n}^2 (\Delta_2 \widetilde{a}_{m,n})^2 \Bigg) \Bigg).
\end{align*}
After exchanging the order of summation and replacing $\widetilde{f}$ and $\widetilde{a}$ with their corresponding definitions, this term takes the form
\begin{align*}
  \mathbb{E} \Bigg( \sum_{n=1}^N \mathbb{E} \Bigg( \sum_{m=1}^M (\Delta_1 f_{m,n})^2 (\Delta_2 a_{m-1,n})^2 \Bigg) \Bigg).
\end{align*}
Define now the filtration
\begin{align*}
  \mathcal{G}_m^2 = \sigma\Bigg( \bigcup_{n=0}^\infty \mathcal{F}_{m,n} \Bigg).
\end{align*}
If we fix $n$, then the sequences $f_{m,n}$ and $\Delta_2 a_{m,n}$ are martingales with respect to this filtration.
Thus, by Theorem \ref{Theorem1}, we obtain
\begin{align*}
  \mathbb{E} \Bigg( \sum_{n=1}^N \mathbb{E} \Bigg( \sum_{m=1}^M (\Delta_1 f_{m,n})^2 (\Delta_2 a_{m-1,n})^2 \Bigg) \Bigg) 
  &\lesssim \mathbb{E}\Bigg( \sum_{n=1}^N f_{M,n}^2 (\Delta_2 a_{M,n})^2 + \sum_{m=1}^M \sum_{n=1}^N (f_{m,n}^2 + f_{m-1,n}^2) (\Delta a_{m,n})^2 \Bigg) \\
  &\lesssim \mathbb{E}\Bigg( \sum_{n=1}^N f_{M,n}^2 (\Delta_2 a_{M,n})^2 + \sum_{m=1}^M \sum_{n=1}^N (f_{m,n}^*)^2 (\Delta a_{m,n})^2 \Bigg).
\end{align*}
It is clear that something very similar can be done with the last term, which has $\widetilde{f}_{m,n-1}$. Putting all the estimates together:
\begin{align*}
  &\mathbb{E}\Bigg( \sum_{m=1}^M \sum_{n=1}^N (\Delta f_{m,n})^2 a_{m-1,n-1}^2 \Bigg) \lesssim \\
  &\mathbb{E}\Bigg( f_{M,N}a_{M,N}^2 +  
  \sum_{m=1}^N (f_{N,m}^*)^2 (\Delta_2 a_{N,m})^2 
  +\sum_{n=1}^N (f_{M,n}^*)^2 (\Delta_1 a_{M,n})^2 
  + \sum_{m=1}^M \sum_{n=1}^N (f_{m,n}^*)^2 (\Delta a_{m,n})^2
  \Bigg).
\end{align*}

We have found an upper bound consisting of the expectation of the sum of four terms.
If one looks closely, these four terms are sums over what could be considered ``boundaries'' various codimensions of the integer lattice box
\begin{align*}
  \{(m,n) \in \mathbb{Z}^2:\, 1 \leq m \leq M \text{ and } 1 \leq n \leq N\}.
\end{align*}
In particular, the first term is just the upper-right-hand corner of the box, the middle two terms are sums over the upper horizontal edge and right vertical edge, and
the last term is the sum over the whole box.

We can write this upper bound more neatly by introducing some notation.
For a subset of indices $\mathcal{I} \subseteq \{1,2\}$, define
\begin{align*}
  \partial_{\mathcal{I}}(M_1,M_2) = \{(m_1,m_2) \in \mathbb{Z}^2:\, 1 \leq m_1 \leq M_1,\, 1 \leq m_2 \leq M_2,\, m_i = M_i \text{ for all $i \in \mathcal{I}$} \}
\end{align*}
In particular,
\begin{align*}
  \partial_{\emptyset}(M_1, M_2) &= \{(M_1, M_2)\} \\
  \partial_{\{1\}}(M_1, M_2) &= \{(1,M_2), (2, M_2), \dots, (M_1,M_2)\}, \\
  \partial_{\{2\}}(M_1, M_2) &= \{(M_1,1), (M_1, 2), \dots, (M_1, M_2)\}, \\
  \partial_{\{1,2\}}(M_1, M_2) &= \{(m,n) \in \mathbb{Z}^2:\, 1 \leq m_1 \leq M_1 \text{ and } 1 \leq m_2 \leq M_2\}.
\end{align*}
Define also
\begin{align*}
  \Delta_{\emptyset} &= \text{Id} \\
  \Delta_{\{1,2\}} &= \Delta_1 \circ \Delta_2.
\end{align*}
Then, we have proved
\begin{align*}
  \mathbb{E}\Bigg( \sum_{m=1}^M \sum_{n=1}^N (\Delta f_{m,n})^2 a_{m-1,n-1}^2 \Bigg) \lesssim 
  \mathbb{E}\Bigg( \sum_{\mathcal{I} \subseteq \{1,2\}} \sum_{(m,n) \in \partial_{\mathcal{I}}(M_1,M_2)} (f^*_{m,n})^2 (\Delta_{\mathcal{I}} a_{m,n})^2 \Bigg).
\end{align*}

\subsection{Arbitrarily many parameters}

In the multiparameter setting, filtrations are indexed by $k$-tuples of non-negative integers.
For simplicity, we will let $\mathbb{N}$ include $0$, and denote by $\mathbb{N}^k$ the cartesian product of $k$ copies of $\mathbb{N}$.

For any integer $n \geq 1$, we will use $[n]$ to denote the set $\{1, 2, \dots, n\}$.
For any $i \in [k]$ let $\pi_i$ be the projection acting on $k$-tuples which ignores everything except the $i$-th coordinate.
In particular, if $m = (m_1, \dots, m_k)$ then $\pi_i(m) = m_i$.

Let $N_i(m)$ be the the $k$-tuple formed by incrementing the $i$-th coordinate by $1$:
\begin{align*}
  \pi_j(N_i(m)) = \begin{cases}
    m_i + 1 &\text{if } j = i \\
    m_i &\text{if } j \neq i
  \end{cases} \quad \text{for all }j \in [k].
\end{align*}

A $k$-parameter filtration is a collection of $\sigma$-algebras $\mathcal{F}_{m} \subseteq \Sigma$ indexed by $m \in \mathbb{N}^k$ which is increasing in each coordinate:
\begin{align*}
  \mathcal{F}_{m} \subseteq \mathcal{F}_{N_i(m)} \quad \text{for all $i \in [k]$.}
\end{align*}
We will assume that our filtration satisfies the $k$-parameter analogue of the \eqref{f4} condition of Cairoli-Walsh:
\begin{align} \label{f4} \tag{$F_4$}
  \mathbb{E}[ \mathbb{E}[f \,|\, \mathcal{F}_{a}] \,|\, \mathcal{F}_{b}] = \mathbb{E}[f \,|\, \mathcal{F}_{a \land b}] \quad \text{for all $a, b \in \mathbb{N}^k$}.
\end{align}
Here $a \land b$ is the $k$-tuple defined componentwise by
\begin{align*}
  (a \land b)_i = \min(a_i, b_i)
\end{align*}
for all $i \in [k]$.

A martingale with respect to the $k$-parameter filtration $\mathcal{F}$ is a sequence of integrable functions $f_{m}$ indexed by $m \in \mathbb{N}^k$ such that
$f_{m}$ is $\mathcal{F}_{m}$-measurable and
\begin{align*}
  \mathbb{E}[f_{N_i(m)} \,|\, \mathcal{F}_{m}] = f_{m}
\end{align*}
for all $i \in [k]$ and $m \in \mathbb{N}^k$.

Given a $k$-tuple $m$ with $m_i \geq 1$ define $P_i(m)$ as the same $k$-tuple but with the $i$-th coordinate replaced by $m_i-1$, similarly to $N_i$:
\begin{align*}
  \pi_j(P_i(m)) = \begin{cases}
    m_i - 1 &\text{if } j = i \\
    m_i &\text{if } j \neq i
  \end{cases} \quad \text{for all }j \in [k].
\end{align*}
This leads to a generalization of the backwards shift operator from the one-parameter case:
\begin{align*}
  (B_i f)_{m} = \begin{cases}
    0 &\text{if } m_i = 0 \\
    f_{P_i(m)} &\text{if } m_i \geq 1.
  \end{cases}
\end{align*}
Then, the difference operator along the $i$-th coordiante, denoted by $\Delta_i$, is defined simply by
\begin{align*}
  \Delta_i f = f - B_i f.
\end{align*}
It is easy to see that the backwards shift operators commute, and thus so do the directional difference operators $\Delta_i$.
This commutativity allows us to define the multiple-directional versions of $B$ and $\Delta$. For any set of indices $\mathcal{I} = \{i_1, i_2, \dots, i_{l}\} \subseteq [k]$ define
\begin{align*}
  \Delta_{\mathcal{I}} &= \Delta_{i_1} \circ \Delta_{i_2} \circ \dots \circ \Delta_{i_l}, \\
  B_{\mathcal{I}} &= B_{i_1} \circ B_{i_2} \circ \dots \circ B_{i_l}.
\end{align*}
When $\mathcal{I} = \emptyset$ then we define both $\Delta_{\mathcal{I}}$ and $B_{\mathcal{I}}$ as the identity.

Let us define a partial order $\leq$ on $k$-tuples by
\begin{align*}
  m \leq n \iff m_i \leq n_i \quad \text{for all $i \in [k]$}.
\end{align*}
Then, the integer-lattice box $\{(m_1, \dots, m_k) \in \mathbb{N}^k:\, 1 \leq m_i \leq M_i\}$ can be denoted simply by $\{m \in \mathbb{N}^k:\, 1 \leq m \leq M\}$.
We will be dealing with the ``boundary'' of boxes like these in various codimensions. To this end, let us define, for a subset $\mathcal{I} \subseteq [k]$ and $M \in \mathbb{N}^k$, the
slices
\begin{align*}
  \partial_{\mathcal{I}}(M) = \{m \in \mathbb{N}^k:\, 1 \leq m \leq M \text{ and } m_i = M_i \text{ for all }i \notin \mathcal{I}\}.
\end{align*}
Here we are using $1$ to denote the $k$-tuple consisting of all ones.

Finally, the ``stopped'' Doob's maximal function is defined as
\begin{align*}
f^*_m = \sup_{n \leq m} |f_n|.
\end{align*}
As in the two-parameter case, we could instead use
\begin{align*}
f^b_m = \sup_{m-1 \leq n \leq m} |f_n|,
\end{align*}
but we will not need this sharper version.

With these definitions, we are ready to state the $k$-parameter version of Theorem \ref{Theorem1}.
\begin{theorem}
  Let $\mathcal{F}$ be a $k$-parameter filtration satisfying the \eqref{f4} condition.
  Let $f$ and $a$ be $k$-parameter martingales.
  Suppose furthermore that $f_{m} a_{n} \in L^2$ for every $m,n \in \mathbb{N}^k$.
  Then, for every $M \in \mathbb{N}^k$ we have
  \begin{align} \label{multip::induction}
    \mathbb{E} \Bigg( \sum_{1 \leq m \leq M} (\Delta f_{m})^2 a^2_{m - 1} \Bigg) \lesssim \mathbb{E} \Bigg( \sum_{\mathcal{I} \subseteq [k]} \sum_{m \in \partial_{\mathcal{I}}(M)} (f^*_{m})^2 (\Delta_{\mathcal{I}} a)^2 \Bigg).
  \end{align}
\end{theorem}
\begin{proof}
  We will prove this by induction on the number of parameters $k$.
  Note that Theorem \ref{Theorem1} implies \eqref{multip::induction} when $k=1$. So we take $k \geq 2$ and assume that \eqref{multip::induction} holds for $(k-1)$-parameter martingales.

  We start by writing the left-hand-side as
  \begin{align*}
    \text{LHS} &= \sum_{1 \leq \overline{m} \leq \overline{M}} \mathbb{E} \Bigg( \sum_{j=1}^{M_k} (\Delta f_{(\overline{m}, j)})^2 (Ba)^2_{(\overline{m}, j)} \Bigg) =: \circledast.
  \end{align*}

  For every integer $j \geq 0$ let
  \begin{align*}
    \mathcal{G}_j = \bigcup_{\substack{m \in \mathbb{N}^k \\ m_k = j}} \mathcal{F}_{m}.
  \end{align*}
  The sequence $(\mathcal{G}_j)_{j \geq 0}$ is a one-parameter filtration, and each $\mathcal{G}_j$ is contained in $\Sigma$.

  We can continue by adding the conditional expectation with respect to $\mathcal{G}_0$:
  \begin{align} \label{kp:intermediate1}
    \circledast &= \sum_{1 \leq \overline{m} \leq \overline{M}} \mathbb{E}\Bigg( \mathbb{E} \Bigg( \sum_{j=1}^{M_k} (\Delta f_{(\overline{m}, j)})^2 (Ba)^2_{(\overline{m}, j)} \,\Big|\, \mathcal{G}_0 \Bigg) \Bigg).
  \end{align}
  Now define the sequences
  \begin{align*}
    \widetilde{f}_{m} &= (\Delta_{k-1} \circ \dots \circ \Delta_{1})f_{m}, \\
    \widetilde{a}_{m} &= (B_{k-1} \circ \dots \circ B_{1})a_{m}.
  \end{align*}
  Notice that we can write the inner conditional expectation as follows:
  \begin{align*}
    \mathbb{E} \Bigg( \sum_{j=1}^{M_k} (\Delta f_{(\overline{m}, j)})^2 (Ba)^2_{(\overline{m}, j)} \,\Big|\, \mathcal{G}_0 \Bigg)
    = \mathbb{E} \Bigg( \sum_{j=1}^{M_k} (\Delta_k \widetilde{f}_{(\overline{m},j)})^2 (B_k \widetilde{a})^2_{(\overline{m},j)} \,\Big|\, \mathcal{G}_0 \Bigg).
  \end{align*}
  Since, at this point, we are only summing over $j$, we will drop the dependence on $\overline{m}$ and just write $\widetilde{f}_j$ and $\widetilde{a}_j$ for $\widetilde{f}_{(\overline{m}, j)}$ and $\widetilde{a}_{(\overline{m}, j)}$ respectively.

  The sequences $(\widetilde{f}_{j})_{j \geq 0}$ and $(\widetilde{a}_{j})_{j \geq 0}$ are martingales with respect to the filtration $(\mathcal{G}_j)_{j \geq 0}$.
  So, by the one-parameter case (Theorem \ref{Theorem1} acting on the $k$-th parameter), we have
  \begin{align*}
    \mathbb{E} \Bigg( \sum_{j=1}^{M_k} (\Delta_k \widetilde{f}_{j})^2 (B_k \widetilde{a})^2_{j} \,\Big|\, \mathcal{G}_0 \Bigg) &\lesssim
      \mathbb{E} \Bigg( \widetilde{f}_{M_k}^2 \widetilde{a}_{M_k}^2 + \sum_{j=1}^{M_k} (\widetilde{f}_{j}^2 + B_k \widetilde{f}_j^2) (\Delta_k \widetilde{a})^2_{j} \,\Big|\, \mathcal{G}_0 \Bigg).
  \end{align*}
  Plugging this into \eqref{kp:intermediate1} we obtain
  \begin{align*}
    \text{LHS} &\lesssim \mathbb{E}\Bigg( \sum_{1 \leq \overline{m} \leq \overline{M}} \Bigg( \widetilde{f}_{\overline{m}, M_k}^2 \widetilde{a}_{\overline{m}, M_k}^2
    + \sum_{j=1}^{M_k} (\widetilde{f}_{(\overline{m},j)}^2 + B_k \widetilde{f}_{(\overline{m},j)}^2) (\Delta_k \widetilde{a})^2_{(\overline{m},j)} \Bigg) \Bigg) \\
    &= \mathbb{E}\Bigg( \sum_{1 \leq \overline{m} \leq \overline{M}} \widetilde{f}_{\overline{m}, M_k}^2 \widetilde{a}_{\overline{m}, M_k}^2 \Bigg)
    + \sum_{j=1}^{M_k} \mathbb{E}\Bigg( \sum_{1 \leq \overline{m} \leq \overline{M}} (\widetilde{f}_{(\overline{m},j)}^2 + B_k \widetilde{f}_{(\overline{m},j)}^2) (\Delta_k \widetilde{a})^2_{(\overline{m},j)} \Bigg).
  \end{align*}
  These terms can be now handled by the induction hypothesis.
  Indeed, the first can be written as
  \begin{align*}
    \mathbb{E}\Bigg( \sum_{1 \leq \overline{m} \leq \overline{M}} (\Delta_{[k-1]} f_{\overline{m}, M_k})^2 (B_{[k-1]} a)_{\overline{m}, M_k}^2 \Bigg).
  \end{align*}
  By the induction hypothesis, this is bounded (up to a constant) by
  \begin{align*}
    \mathbb{E}\Bigg( \sum_{\mathcal{I} \subseteq [k-1]} \sum_{m \in \partial_{\mathcal{I}}(M)} (f_{m}^*)^2 (\Delta_{\mathcal{I}} a_{m})^2 \Bigg).
  \end{align*}

  The second term splits into two parts:
  \begin{align*}
    &\sum_{j=1}^{M_k} \mathbb{E}\Bigg( \sum_{1 \leq \overline{m} \leq \overline{M}} \widetilde{f}_{(\overline{m},j)}^2 (\Delta_k \widetilde{a})^2_{(\overline{m},j)} \Bigg) +
    \sum_{j=1}^{M_k} \mathbb{E}\Bigg( \sum_{1 \leq \overline{m} \leq \overline{M}} B_k \widetilde{f}_{(\overline{m},j)}^2 (\Delta_k \widetilde{a})^2_{(\overline{m},j)} \Bigg) \\
    &= \sum_{j=1}^{M_k} S_j + \sum_{j=1}^{M_k} T_j.
  \end{align*}
  Again, we can handle $S_j$ and $T_j$ with the induction hypothesis.
  Recall that $\Delta$ commutes with $B$ (over any indices), so we can write $S_j$ as
  \begin{align*}
    S_j = \mathbb{E}\Bigg( \sum_{1 \leq \overline{m} \leq \overline{M}} (\Delta_{[k-1]} f_{\overline{m}, j})^2 (B_{[k-1]}(\Delta_k a)_{\overline{m}, j})^2 \Bigg).
  \end{align*}
  Thus, by the induction hypothesis,
  \begin{align*}
    S_j \lesssim \mathbb{E}\Bigg( \sum_{\mathcal{I} \subseteq [k-1]} \sum_{\overline{m} \in \partial_{\mathcal{I}}(\overline{M})} (f_{\overline{m}, j}^*)^2 (\Delta_{\mathcal{I}} \Delta_k a_{\overline{m}, j})^2  \Bigg).
  \end{align*}
  Summing over $j$:
  \begin{align*}
    \sum_{j=1}^{M_k} S_j &\lesssim \mathbb{E}\Bigg( \sum_{j=1}^{M_k}\sum_{\mathcal{I} \subseteq [k-1]} \sum_{\overline{m} \in \partial_{\mathcal{I}}(\overline{M})} (f_{\overline{m}, j}^*)^2 (\Delta_{\mathcal{I}} \Delta_k a_{\overline{m}, j})^2 \Bigg) \\
    &= \mathbb{E}\Bigg( \sum_{j=1}^{M_k}\sum_{\mathcal{I} \subseteq [k-1]} \sum_{\overline{m} \in \partial_{\mathcal{I}}(\overline{M})} (f_{\overline{m}, j}^*)^2 (\Delta_{\mathcal{I} \cup \{k\}} a_{\overline{m}, j})^2 \Bigg) \\
    &= \mathbb{E}\Bigg( \sum_{\substack{\mathcal{I} \subseteq [k] \\ \mathcal{I} \ni k}} \sum_{m \in \partial_{\mathcal{I}}(M)} (f_{m}^*)^2 (\Delta_{\mathcal{I}} a_{m})^2 \Bigg) \\
    &\leq \mathbb{E}\Bigg( \sum_{\mathcal{I} \subseteq [k]} \sum_{m \in \partial_{\mathcal{I}}(M)} (f_{m}^*)^2 (\Delta_{\mathcal{I}} a_{m})^2 \Bigg).
  \end{align*}
  The terms $T_j$ are handled exactly in the same way. The only thing one needs to notice is that $B_k f$ is a martingale with respect to the first $(k-1)$ parameters. Then
  \begin{align*}
    \sum_{j=1}^{M_k} T_j &\lesssim \mathbb{E}\Bigg( \sum_{j=1}^{M_k}\sum_{\mathcal{I} \subseteq [k-1]} \sum_{\overline{m} \in \partial_{\mathcal{I}}(\overline{M})} (f_{\overline{m}, j-1}^*)^2 (\Delta_{\mathcal{I}} \Delta_k a_{\overline{m}, j})^2 \Bigg) \\
    &\leq \mathbb{E}\Bigg( \sum_{j=1}^{M_k}\sum_{\mathcal{I} \subseteq [k-1]} \sum_{\overline{m} \in \partial_{\mathcal{I}}(\overline{M})} (f_{\overline{m}, j}^*)^2 (\Delta_{\mathcal{I}} \Delta_k a_{\overline{m}, j})^2 \Bigg) \\
    &\leq \mathbb{E}\Bigg( \sum_{\mathcal{I} \subseteq [k]} \sum_{m \in \partial_{\mathcal{I}}(M)} (f_{m}^*)^2 (\Delta_{\mathcal{I}} a_{m})^2 \Bigg).
  \end{align*}
\end{proof}

\section{Brossard's inequality}
\label{section.brossard}
In this section we prove Theorem \ref{TheoremC}. That is,
for a regular filtration $\mathcal{F}_{m}$ satisfying the \eqref{f4} condition, we have
\begin{align*}
P(Sf> \lambda) \lesssim P(E) + \lambda^{-2}\mathbb{E}[(f^*)^2;\, E^c],
\end{align*}
where $E = \{f^* > \lambda\}$.

We will assume that $f$ is ``stopped'', i.e.: $\Delta f_m = 0$
for all $m$ outside a sufficiently large box $m \leq M$.
Moreoever, by subtracting off the relevant terms in $\sum_m \Delta f_m$,
we may assume that $\Delta f_m$ also vanishes when $m_i = 0$ for any $i \in [k].$

Let $R$ be the regularity constant of the filtration. Recall that we then have
\begin{align*}
a_{N_i(m)} \leq R a_{m}
\end{align*}
for all non-negative martingales $a$.
In particular, for the product dyadic filtration, $R = 2$.

Define the \emph{enlargement} of any measurable set $E$ as
\begin{align*}
\Enl{E} = \{\ind_E^* > R^{-k-1}\}.
\end{align*}
One can check that $P(\Enl(E)) \lesssim P(E)$, where the implicit constant depends only on $k$ and $R$.

Define the martingale
\begin{align*}
a_m = \mathbb{E}\big( \ind_{\Enl(E)^c} \,|\, \mathcal{F}_m \big).
\end{align*}
This will serve as a ``stopping bump function'' for $f$. First, note that
$f^*_m \leq \lambda$ on the support of $a^*$. In particular, for $m \geq 1$
\begin{align*}
f^*_m > \lambda \implies \Delta a_m = 0.
\end{align*}
This follows from the regularity of the filtration. Indeed, if $|f_k(x)| > \lambda$
then $E$ contains an $\mathcal{F}_k$-measurable set containing $x$.
Enlarging the set (by the regularity condition) will increase the measurability of this set, and in particular $\Enl(E)$ will contain an $\mathcal{F}_{k-1}$-measurable set containing $x$. By definition, we then have $a_{m-1} = 0$ on $x$, and thus $\Delta a_m = 0$ using that $a$ is nonnegative.

We now proceed in the standard fashion:
\begin{align*}
P(Sf > \lambda) \leq P(\Enl^2(E)) + \lambda^{-2}\int_{F^c} \sum_{1 \leq m \leq M} (\Delta f_m)^2.
\end{align*}
The first term is bounded by $P(E)$, so we just focus on the second term.

Note that, on $F^c$, we have $a_{m-1} \gtrsim 1$. Again, this follows from the definition of the enlargement operator and the regularity of the filtration. This allows us to bound
\begin{align*}
\int_{F^c} \sum_{1 \leq m \leq M} (\Delta f_m)^2 &\lesssim \int \sum_{1 \leq m \leq M} (\Delta f_m)^2 a_{m-1}^2.
\end{align*}

We are now ready to apply Theorem \ref{TheoremB}:
\begin{align*}
\int \sum_{1 \leq m \leq M} (\Delta f_m)^2 a_{m-1}^2 &\lesssim
\int (f_M^*)^2 a_M^2 + \sum_{\substack{\mathcal{I} \subseteq [k]\\ \mathcal{I} \neq \emptyset}} \int \sum_{m \in \partial_{\mathcal{I}}(M)} (f_m^*)^2 (\Delta_{\mathcal{I}} a_m)^2 \\
&\lesssim 
\int (f_M^*)^2 a_M^2 + \lambda^{2}\sum_{\substack{\mathcal{I} \subseteq [k]\\ \mathcal{I} \neq \emptyset}} \int \sum_{m \in \partial_{\mathcal{I}}(M)}  (\Delta_{\mathcal{I}} a_m)^2.
\end{align*}
For the first term, note that $E$ is $\mathcal{F}_M$-measurable, so $a_M = \ind_{\Enl(E)^c}$ and we have
\begin{align*}
\int (f_M^*)^2 a_M^2 \leq \int_{f^* \leq \lambda} (f^*)^2.
\end{align*}

For the second term, observe that we are skipping the case of $\mathcal{I} = \emptyset$.
This means that we can subtract from $a$ an arbitrary constant since $\Delta_{\mathcal{I}}(1)_m = 0$ when $m \geq 1$. Fixing $\mathcal{I} \neq \emptyset$ we have, using that the square function is an $L^2$ isometry,
\begin{align*}
\int \sum_{m \in \partial_{\mathcal{I}}(M)}  (\Delta_{\mathcal{I}} a_m)^2 &= 
\int \sum_{m \in \partial_{\mathcal{I}}(M)}  (\Delta_{\mathcal{I}} (1-a)_m)^2 \\
&\leq \int (1-a_M)^2 = P(\Enl(E)) \lesssim P(E),
\end{align*}
which is what we wanted.

\bibliography{bibliography}
\bibliographystyle{abbrv}

\end{document}